\documentclass[11pt]{article}
\usepackage[intlimits]{amsmath}
\usepackage{amsfonts, amssymb, amsthm}
\usepackage{esint}
\usepackage{epsfig}
\usepackage{color}
\usepackage{threeparttable} 

\renewcommand{\O}{\Omega}
\newcommand{\po}{\partial\O}

\renewcommand{\a}{\alpha}

\renewcommand{\div}{\,\mathrm{div}\,}

\newcommand{\s}{\sigma}

\newtheorem{theorem}{Theorem}[section]
\newtheorem{thm}{Theorem}[section]
\newtheorem{prop}[thm]{Proposition}
\newtheorem{cor}[thm]{Corollary}
\newtheorem{lemma}[thm]{Lemma}
\newtheorem{defn}[thm]{Definition}

\newtheorem{preremark}[thm]{Remark}
\newenvironment{remark}{\begin{preremark}\rm}{\medskip \end{preremark}}
\numberwithin{equation}{section}

\newcommand\res{\hbox{ {\vrule height .22cm}{\leaders\hrule\hskip.2cm} } }

\newcommand{\R}{\mathbb R}

\def\XXint#1#2#3{{\setbox0=\hbox{$#1{#2#3}{\int}$}
       \vcenter{\hbox{$#2#3$}}\kern-.5\wd0}}

\newcommand{\meanbar}[1]{%
\setbox0 = \hbox{$#1 \int$}
\hbox to 0pt{%
\thinspace
\hskip 0.1\wd0
\raise 0.5\ht0
\hbox{%
\lower 0.5\dp0
\hbox{\rule{0.8\wd0}{2\linethickness}}
}%
\hss
}%
}

    \newcounter{myfootertablecounter}

\begin{document}
\title{Perturbations of elliptic operators in chord arc domains.}

\author{Emmanouil Milakis, Jill Pipher, Tatiana Toro}
\date{}
\maketitle

\begin{abstract}

We study the boundary regularity of solutions to divergence form operators which are small perturbations of operators for which the boundary regularity of solutions is known.
An operator is a small perturbation of another operator if the deviation function of the coefficients satisfies a Carleson measure condition with small norm. 
We extend Escauriaza's result on Lipschitz domains to chord arc domains with small constant. In particular we prove that if $L_1$ is a small perturbation of $L_0$ and $\log k_0$ 
has small BMO norm so does $\log k_1$. Here $k_i$ denotes the density of the elliptic measure of $L_i$ with respect to the surface measure of the boundary of the domain.

AMS Subject Classifications: 35J25, (31B05) \\ \textbf{Keywords}:
Chord arc domain, Elliptic measure, VMO.

\end{abstract}

\section{Introduction}
\renewcommand{\thesection}{\arabic{section}}

In this paper we study the regularity properties of the elliptic measure associated to an elliptic operator in divergence form, $L=\text{div}A\nabla$ 
on chord arc domains (CADs). We assume that $A$ is a small perturbation of the matrix associate to a \emph{regular} operator. See discussion below for the definition of 
small perturbation and the notion of \emph{regular} operator. Chord arc domains are not necessary Lipschitz domains, in general they cannot be 
locally represented as graphs. This lack of a \emph{preferred} direction even at the local level introduces a new set of challenges.
On the other hand their geometry is sufficiently under control in order to 
develop and use tools from harmonic analysis.  Chord arc domains in $\R^n$ are non-tangentially accessible (NTA) domains whose boundaries are Ahlfors regular (a "non-
degeneracy" condition indicating that the surface measure of $(n-1)$-dimensional balls with center on the boundary and radius $r$ should behave like $r^{n-1}$). 
CADs are sets of locally finite perimeter (see \cite{eg}).
In \cite{kt1}, Kenig and 
Toro showed that if $\Omega$ is a $(\delta,R)-$ chord arc domain with small $\delta$ (see Definition \ref{CADdr} below), then the unit 
normal to $\partial\Omega$ has small BMO constant with respect to $\sigma=\mathcal{H}^{n-1}\res\partial\Omega$ the surface measure to $\partial\Omega$.

For the Laplace operator, $L=\Delta$, Dahlberg \cite{d1} proved that if $\Omega$ is a strongly Lipschitz domain then the harmonic measure and the surface measure are mutually 
absolutely continuous and the Poisson kernel is in $L^2(\sigma)$. In \cite{jk82}, Jerison and Kenig showed that if $\Omega$ is a $C^1$ domain then $\log k$ (the logarithm of the 
Poisson kernel) belongs to $\text{VMO}(\sigma)$ where VMO is the Sarason space of vanishing mean oscillation. In \cite{kt1}, Kenig and Toro extended this result to a non-smooth 
setting by proving that if $\Omega$ is a chord arc domain with vanishing constant (see Definition \ref{vanishing} below) then $\log k$ belongs to $\text{VMO}(\sigma)$. 

Questions concerning the regularity of the elliptic measure for variable coefficients operators in divergence form are rather delicate as was shown by the work of \cite{cfk} and
\cite{mm} where examples of operators with singular elliptic measures with respect to surface measure on smooth domains were constructed.
Regularity results have been obtained, on Lipschitz domains, provided that the coefficient matrix $A$ is given as a perturbation of a 
given matrix $A_0$ that corresponds to an elliptic operator whose elliptic measure is \emph{regular} with respect to the surface measure to the boundary. 
In \cite{d4}, Dahlberg introduced the notion of perturbation of elliptic operators in Lipschitz 
domains. Roughly speaking an operator $L=\div A\nabla$  is a perturbation of an operator $L_0=\div A_0\nabla$,
if the difference between the coefficient matrices $A$ and $A_0$ satisfies a Carleson condition.

More precisely, let $\Omega\subset\R^{n}$ be a CAD (see Definition \ref{CAD}) and consider two elliptic operators $L_i={\rm{div}}(A_i\nabla\ )$ for $i=0,1$ in $\Omega$. We say that $L_1$ is a perturbation of $L_0$ if the deviation function
\begin{equation}\label{eqn:tt-a}
a(X)=\sup\{|A_1(Y)-A_0(Y)|: Y\in B(X,\delta(X)/2)\}
\end{equation}
where
$\delta(X)$ is the distance of $X$ to $\partial \Omega$, satisfies the following Carleson measure property: there exists a constant $C>0$ such that
\begin{equation}\label{normfkp}
\sup_{0<r<\rm{diam}\Omega}\sup_{Q\in\partial \Omega}
\bigg\{\frac{1}{\sigma(B(Q,r))}\int_{B(Q,r)\cap \Omega}\frac{a^2(X)}{\delta(X)}dX\bigg\}^{1/2}\le
 C.
\end{equation}

Note that in this case $L_1=L_0$ on $\partial \Omega$. $L_1$ is said to be a perturbation of $L_0$ with vanishing Carleson constant if for each compact $K\subset \mathbb{R}^n$
\begin{equation}\label{carlesonVanish}
\lim_{r\rightarrow 0}\gamma_K(r)=0
\end{equation}
where
\begin{equation}\label{carleson1}
\gamma_K(r)=\sup_{Q\in \partial \Omega \cap K}\sup_{0<s\leq r}\bigg(\frac{1}{\sigma(B(Q,s))}\int_{B(Q,s)\cap \Omega}\frac{a^2(X)}{\delta(X)}dX\bigg)^{1/2}.
\end{equation}

For $i=0,1$ we denote by $G_i(X,Y)$  the Green's function of $L_i$ in $\Omega$ with pole at $X$ and by $\omega^X_i$ the corresponding elliptic measure. Since the results below 
are independent of the pole $X$ to simplify notation we denote by $\omega_i$ the elliptic measure of $L_i$. Recall that $k_i$ is the Radon-Nikodym derivative of $\omega_i$ with respect to $\sigma$.

\begin{theorem}[Dahlberg \cite{d4}]
Let $\Omega\subset\R^n$ be a Lipschitz domain. Assume that $L_1$ is a perturbation of $L_0$ with vanishing Carleson constant then $\omega_0\in B_p(\sigma)$ for some $p\in (1,\infty)$ if and only if $\omega_1\in B_p(\sigma)$.
\end{theorem}

\begin{theorem}[Fefferman-Kenig-Pipher \cite{fkp}]
Let $\Omega\subset\R^n$ be a Lipschitz domain. Assume that $L_1$ is a perturbation of $L_0$ then $\omega_0\in A_\infty(\sigma)$ if and only if $\omega_1\in A_\infty(\sigma)$.
\end{theorem}

\begin{theorem}[Escauriaza \cite{esc1}]
Let $\Omega\subset\R^n$ be a Lipschitz domain. Assume that $L_1$ is a perturbation of $L_0$ with vanishing Carleson constant then $\log k_0\in \text{VMO}(\sigma)$ if and only if
$\log k_1\in \text{VMO}(\sigma)$.
\end{theorem}

Dahlberg's proof is based on a very original idea. He shows that a differential inequality holds for a quantity that controls the $B_p(\sigma)$ norms of the elliptic measures 
a one parameter family of operators that interpolate between $L_0$ and $L_1$. Escauriaza builds on this idea.  On the other hand  Fefferman, Kenig and Pipher use a completely 
different approach based on harmonic analysis techniques.
 
In this paper we extend Escauriaza's result to the CAD setting. The present work is a natural  
continuation of \cite{mt1} and \cite{MPT}. Although we follow Escauriaza's road map, the justification
of most steps depend on arguments that resemble those of \cite{MPT} which required developing harmonic analysis
techniques on CADs. The recurrent theme is that
since CAD are not locally representable as the graph of a good function we need to appeal to their 
geometry and the Ahlfors regality property of their boundary. In \S 2 we summarize some of the results from \cite{MPT}
and combine them with classical results from the theory of weights. In particular Corollary \ref{surface} guarantees that 
we can proceed as in Escauriaza's (see Remark \ref{esch}). In \S 3 we prove the main result which reduces to a differential 
inequality which yields as in Dahlberg's and Escauriaza's case a bound for the appropriate $B_p(\sigma)$ norm.

\section{Preliminaries}\label{prelims}
\renewcommand{\thesection}{\arabic{section}}

Let $\Omega\subset \R^{n}$ be a bounded
domain and consider elliptic
operators $L$ of the form
$
Lu=\textrm{div}(A(X)\nabla u)
$
defined in $\Omega$ where $A(X)=(a_{ij}(X))$ is a symmetric matrix such that there
are $\lambda, \Lambda>0$ satisfying
\begin{equation}\label{ellipticity}
\lambda |\xi|^2\le
 \sum_{i,j=1}^{n}a_{ij}(X)\xi_i\xi_j\le
 \Lambda |\xi|^2 \qquad \hbox{ for all } X\in\Omega\hbox{ and }\xi \in \R^{n}.
\end{equation}

We say that a function $u$ in $\Omega$ is a solution to $Lu=0$ in $\Omega$ provided that $u\in W_{\rm{loc}}^{1,2}(\Omega)$ and for all $\phi\in C^{\infty}_c(\Omega)$,
$\int_{\Omega}\langle A(x)\nabla u,\nabla\phi\rangle dx =0.$
A domain $\Omega$ is called regular for the operator $L$, if for
every $g\in C(\partial \Omega)$, the generalized solution of the
classical Dirichlet problem  with boundary data $g$ is a function $u\in C(\overline{\Omega})$. Let $\Omega$ be a regular domain for $L$, the Riesz Representation Theorem ensures that there exists a family of regular Borel probability measures
$\{\omega^X_L\}_{X\in\Omega}$ such that the function
$u(X)=\int_{\partial\Omega}g(Q)d\omega^X_L(Q)$ satisfies the Dirichlet problem for $L$ with boundary data $g$.
For $X\in \Omega$, $\omega^X_L$ is called the $L-$elliptic measure of $\Omega$ with pole $X$. When no
confusion arises, we will omit the reference to $L$ and simply
called it as the elliptic measure. 

The following lemmas contain some properties on the boundary behavior of $L-$elliptic solutions in non-tangentially accessible (NTA) domains for uniformly elliptic divergence form operators $L$ with bounded measurable coefficients. We refer the reader to \cite{jk}, \cite{k1} for the definitions and more details regarding elliptic operators of divergence form
defined in NTA domains.

\begin{lemma}[Cacciopoli Inequality]\label{Cacciop}
Let $u$ be a non-negative subsolution in $\Omega$ and $\overline{B(X,2R)}\subset \Omega$. Then
$${\fint}_{B(X,R)}|\nabla u(X)|^2dX\leq \frac{C}{R^2}{\fint}_{B(X,2R)}u^2(X)dX$$
where constant $C$ depends on the ellipticity constants $\lambda,\Lambda$ and the dimension $n$.
\end{lemma}

\begin{lemma}[Boundary Cacciopoli Inequality]\label{bdCacciop}
Let $\Omega$ be an NTA domain and $Q\in\partial\Omega$. If $u$ satisfies $Lu=0$ in $T(Q,4R)=B(Q,4R)\cap\Omega$ and $u=0$ on $\Delta(Q,4R)=B(Q,4R)\cap\partial\Omega$ then
$${\fint}_{T(Q,R)}|\nabla u(X)|^2dX\leq \frac{C}{R^2}{\fint}_{T(Q,2R)}u^2(X)dX$$
where constant $C$ depends on the ellipticity constants $\lambda,\Lambda$ and the dimension $n$.
\end{lemma}

\begin{lemma}\label{lem2.3}
Let $\Omega$ be an NTA domain, $Q\in\partial\Omega$, $0<2r<R$, and
$X\in\Omega\backslash B(Q,2r)$. Then
$$
C^{-1}<\frac{\omega^{X}(B(Q,r))}{r^{n-2}G(A(Q,r),X)}<C,
$$
where $G(A(Q,r),X)$ is the $L-$Green function of $\Omega$ with pole $X$, $\omega^X$ is the corresponding elliptic measure and $A(Q,r)$ is a non-tangential point for $Q$ at $r$.
\end{lemma}

\begin{lemma}[Comparison Principle]\label{lem2.5}
Let $\Omega$ be an NTA domain and let $M>1$ be such that $0<Mr<R$. Suppose that $u,v$ vanish continuously on
$\partial\Omega\cap B(Q,Mr)$ for some $Q\in\partial\Omega$, $u,v\ge
 0$ and
$Lu=Lv=0$ in $\Omega\cap B(Q, Mr)$. Then for
all $X\in B(Q,r)\cap\Omega$,
$$
C^{-1}\frac{u(A(Q,r))}{v(A(Q,r))}\le
 \frac{u(X)}{v(X)}\le
C\frac{u(A(Q,r))}{v(A(Q,r))}
$$
where the
constant $C>1$ only depends on the dimension, the NTA constants and the ellipticity constants.
\end{lemma}

An immediate consequence of the previous lemma is the following boundary regularity result.

\begin{lemma}[H\"{o}lder Regularity]\label{lem2.6}
Let $u,v$ be as in Lemma \ref{lem2.5}, then there exists $\vartheta\in (0,1)$ such that
$$
\bigg|\frac{u(Y)}{v(Y)}-\frac{u(X)}{v(X)}\bigg|\leq \frac{u(A_r(Q))}{v(A_r(Q))}\bigg(\frac{|X-Y|}{r}\bigg)^\vartheta
$$
for all $X,Y\in  B(Q,r)\cap\Omega$. Here $\vartheta$ depends on on the dimension, the NTA constants and the ellipticity constants.
\end{lemma}

\begin{defn}\label{CAD}
We say that  $\Omega \subset \R^{n}$ is a chord arc domain (CAD) if $\Omega$ is an NTA
domain whose boundary is Ahlfors regular, i.e. the surface measure to the boundary satisfies the following condition: there exists $C> 1$ so that for $r\in (0,{\rm{diam}}\,\Omega)$ 
and $Q\in\partial \Omega$
\begin{equation}\label{1.7A}
C^{-1}r^{n-1}\le
\sigma(B(Q,r))\le
 Cr^{n-1}.
\end{equation}
Here $B(Q,r)$ denotes the $n$-dimensional ball of radius
$r$ and center $Q$ and $\sigma=\mathcal{H}^{n-1}\res \partial\Omega$ and $\mathcal{H}^{n-1}$ denotes the $(n-1)$-dimensional Hausdorff measure. The best constant $C$ above is referred to as the Ahlfors regularity constant.
\end{defn}

As mentioned earlier CAD are sets of locally finite perimeter (see \cite{eg}).
Let $\Omega\subset \R^n$ be a domain. Let D denote Hausdorff distance between closed sets. We define
\begin{equation}
\theta(r) = \sup_{Q \in \partial\Omega}\inf_{\mathcal L}
r^{-1} D[\partial\Omega \cap B(Q,r), \mathcal L \cap B(Q,r)],
\end{equation}
where the infimum is taken over all $(n-1)$-planes containing $Q\in\partial\Omega$.

\begin{defn}\label{CADdr}
Let $\Omega \subset \R^{n}$ be a bounded domain, $\delta> 0$ and $R>0$. 
We say that
$\Omega$ is a $(\delta,R)$-chord arc domain (CAD) if $\Omega$ is a
set of locally finite perimeter such that
\begin{equation}\label{1.6}
\sup_{0<r<
 R}\theta(r)\le
 \delta
\end{equation}
and
\begin{equation}\label{1.7}
\sigma( B(Q,r))\le
 (1+\delta)\omega_{n-1}r^{n-1}\ \ \forall
Q\in\partial\Omega\ \ {\rm{and}} \ \forall r\in (0,R).
\end{equation}
Here $\omega_{n-1}$ is the volume of the $(n-1)$-dimensional unit
ball in $\R^{n-1}$.
\end{defn}

\begin{defn}\label{vanishing}
Let $\Omega \subset \R^{n}$, we say that $\Omega$ is a chord arc
domain with vanishing constant if it is a $(\delta, R)$-CAD for
some $\delta>0$ and $R>0$,
\begin{equation}\label{1.8}
\limsup_{r\rightarrow 0}\theta(r)=0
\end{equation}
and
\begin{equation}\label{1.9}
\lim_{r\rightarrow 0}\sup_{Q\in\partial \Omega}\frac{\sigma(B(Q,r))}{\omega_{n}r^{n-1}} =1.
\end{equation}
\end{defn}



Next we recall some fine properties concerning perturbations of elliptic operators in CAD (see (\ref{normfkp}), (\ref{carlesonVanish}) (\ref{carleson1}) for the relevant definitions). In \cite{MPT}, we showed that we may assume $a(X)=0$ in $\Omega$ for $X\in\Omega$ with $\delta(X)>4R_0$ where $R_0=\frac{1}{2^{30}}\min\{\delta(0),1\}$ and  $0\in \Omega$. In particular we cover the boundary $\partial\Omega$ by balls $\{B(Q_i,R_0/2)\}_{i=1}^M$ such that $Q_i\in\partial\Omega$ and $|Q_i-Q_j|\ge
 \frac{R_0}{2}$ for $i\neq j$ and consider the partition of unity $\{\varphi_i\}_{i=1}^M$ associated with this covering that satisfies $0\le
\varphi_i\le
 1$, ${\rm{spt}}\varphi_i\subset B(Q_i,2R_0)$, $\varphi_i\equiv 1$ on $B(Q_i,R_0)$ and $|\nabla\varphi_i|\le
 4/R_0$. Then if we define
$$\psi_i(X)=\left\{
  \begin{array}{ll}
    \bigg(\sum_{j=1}^M\varphi_j(X)\bigg)^{-1}\varphi_i(X) & {\rm{if}} \  \sum_{j=1}^M\varphi_j(X)\neq 0\\
    0 & {\rm{otherwise}}.
  \end{array}
\right.$$
and 
 \begin{equation}\label{tt-a'}
 A'(X)=\bigg(\sum_{j=1}^{M}\psi_j(X)\bigg)A_1(X)+\bigg(1-\sum_{j=1}^{M}\psi_j(X)\bigg)A_0(X)
 \end{equation}
 the following lemmas hold.
\begin{lemma}[\cite{MPT}]\label{claim1-p4}
Let $A'$ be as in (\ref{tt-a'}) then for $X\in \Omega$, with $\delta(X)>4R_0$,
$$a'(X)=\sup_{B(X,\delta(X)/2)}|A'(Y)-A_0(Y)|=0.$$
\end{lemma}
\begin{lemma}[\cite{MPT}]\label{claim2-p4}
If $\omega'$ denotes the elliptic measure associated to $L'={\rm{div}}A'\nabla$ with pole at $0$, then $\omega_1\in B_p(\omega_0)$ if and only if $\omega'\in B_p(\omega_0)$. Here 
we assume that both $\omega_0$ and $\omega_1$ have pole at 0.
\end{lemma}

One of the main results in \cite{MPT} concerns the regularity of the elliptic measure of perturbation operators in CADs. In particular it was shown that if a Carleson norm of the deviation function (see \ref{eqn:tt-a}) is small  then "good" properties of the elliptic measure are preserved. 

\begin{thm}[\cite{MPT}]\label{mainthm1MPT}
Let $\Omega$ be a CAD, $0\in\Omega$ and $\omega_0$, $\omega_1$ are the elliptic measures associated with $L_0$ and $L_1$ respectively with pole 0. There exists 
$\varepsilon_0 >0$, depending also on the ellipticity constants, the dimension, the CAD constants such that if
\begin{equation}\label{condThm2.11}
\sup_{\Delta \subseteq \partial \Omega}\bigg\{\frac{1}{\omega_0(\Delta)}\int_{T(\Delta)}a^2(X)\frac{G_0(X)}{\delta^2(X)}dX\bigg\}^{1/2}\le
 \varepsilon_0\quad then\quad \omega_1\in B_2(\omega_0).
\end{equation}
Here $T(\Delta)=B(Q,r)\cap \Omega$ is the tent associated to the surface ball $\Delta=\Delta_r(Q)=B(Q,r)\cap \partial \Omega$  and $G_0(X)=G_0(0,X)$ denotes the Green's function for $L_0$ in $\Omega$ with pole at $0$.
\end{thm}

Note that (\ref{condThm2.11}) and the Carleson measure property (\ref{normfkp}) relate as follows.

\begin{prop}[\cite{MPT}]\label{p-fkp}
Let $\Omega$ be a CAD and that assume $\omega_0\in B_p(\sigma)$ for some $p>1$. 
Given $\epsilon>0$ there exists $\delta>0$ such that if
\begin{equation}\label{small-carl}
\sup_{\Delta \subseteq \partial \Omega}
\bigg\{\frac{1}{\sigma(\Delta)}\int_{T(\Delta)}\frac{a^2(X)}{\delta(X)}dX\bigg\}^{1/2}\le
 \delta,
\end{equation}
then
\begin{equation}\label{small-fkp}
\sup_{\Delta \subseteq \partial \Omega}\bigg\{\frac{1}{\omega_0(\Delta)}\int_{T(\Delta)}a^2(X)\frac{G_0(X)}{\delta^2(X)}dX\bigg\}^{1/2}\le
\epsilon.
 \end{equation}
\end{prop}

An immediate consequence of Theorem \ref{mainthm1MPT} deals with the  $L^r(d\sigma)$-integrability of  $k_1=\frac{d\omega_1}{d\sigma}$ provided that a suitable condition is assumed for $\omega_0$.

\begin{thm}[\cite{MPT}]\label{pp'}
Let $\Omega$ be a CAD and $\omega_0$, $\omega_1$ be as in Theorem \ref{mainthm1MPT}. If $\omega_1\in B_p(\omega_0)$ for some $1<p<\infty$ and $\omega_0\in B_q(\sigma)$ then $\omega_1\in B_r(\sigma)$ with $r=\frac{qp}{q+p-1}<q$.
\end{thm}
\begin{proof}Consider $r=\frac{qp}{q+p-1}$ and let $h=d\omega_1/d\omega_0$, $k_0=d\omega_0/d\sigma$ and $k_1=d\omega_1/d\sigma$.
Then
$$\int_\Delta k_1^rd\sigma=\int_\Delta h^rk_0^{r/p}k_0^{r(1-1/p)}d\sigma\leq \bigg(\int_\Delta (h^rk_0^{r/p})^{q/q-(1-1/p)r}d\sigma\bigg)^{\frac{q-(1-1/p)r}{q}}
\bigg(\int_\Delta k_0^qd\sigma\bigg)^{\frac{r(1-1/p)}{q}}$$
that is,
$$\int_\Delta k_1^rd\sigma \leq \bigg(\int_\Delta h^pd\omega_0\bigg)^{\frac{q-(1-1/p)r}{q}}\bigg(\int_\Delta k_0^qd\sigma\bigg)^{\frac{r}{q(1-1/p)}}$$
or by the selection of $r$,
$$\int_\Delta k_1^rd\sigma \leq \bigg(\int_\Delta h^pd\omega_0\bigg)^{\frac{q}{q+p-1}}\bigg(\int_\Delta k_0^qd\sigma\bigg)^{\frac{p-1}{p+q-1}}.$$
Since
$$\int_{\Delta}k_0^qd\sigma \leq \sigma(\Delta) \bigg({\fint}_\Delta k_0 d\sigma\bigg)^q\ \ {\rm{and}}\ \ \int_{\Delta}h^pd\omega_0 \leq
\omega_0(\Delta) \bigg({\fint}_\Delta h d\omega_0\bigg)^p$$
we conclude that
$$\int_\Delta k_1^rd\sigma \lesssim \bigg(\int_\Delta k_1 d\sigma\bigg)^r\sigma(\Delta)^{1-r}$$
or
$$\bigg({\fint}_\Delta k_1^r d\sigma\bigg)^{1/r}\lesssim {\fint}_\Delta k_1d\sigma$$
and the proof is complete since $r=\frac{qp}{q+p-1}<q$.
\end{proof}

Throughout the paper we shall use the notation $a\lesssim b$ to mean that there is a constant $C>0$ such that $a\le Cb$.

A slight improvement of the result in Theorem \ref{mainthm1MPT} can be obtained due to an argument of Gehring (\cite{G}, Lemma 2), see also the book of Grafakos (\cite{Gr}).
\begin{lemma}\label{gehring}
Let $\Omega$ be a CAD and $\omega_0$, $\omega_1$ be as in Theorem \ref{mainthm1MPT}. If condition (\ref{condThm2.11}) is satisfied then there exists a constant $\eta_0>0$ which depends only on the constant $\varepsilon_0$ which appears in (\ref{condThm2.11}), the CAD and ellipticity constants such that $\omega_1\in B_{2(1+\eta_0)}(\omega_0)$.
\end{lemma}

Once we combine Theorem \ref{pp'} along with Lemma \ref{gehring} we obtain the following corollary.

\begin{cor}\label{surface}
 Let $\Omega$ be a CAD and $\omega_0$, $\omega_1$ be as in Theorem \ref{mainthm1MPT}. For $\delta_0>0$ small enough there exists $q_0$ large enough depending only on the CAD constants, the dimension and the ellipticity constants such that if $\omega_0\in B_{q_0}(\sigma)$ then $\omega_1\in B_{2(1+\delta_0)}(\sigma)$.
\end{cor}

In the sequel we denote the area integral and the nontagential maximal function by
 $$S_M(u)(Q)=\bigg(\int_{\Gamma_M(Q)}|\nabla u(X)|^2 {\delta(X)}^{2-n}dX\bigg)^{1/2}\ \ {\rm{and}}\ \ N(u)(Q)=\sup\{|u(X)|: X\in \Gamma_M(Q)\} $$
respectively where for $Q\in\partial\Omega$
\begin{equation}\label{cone-a}
\Gamma_M(Q)=\{X \in \Omega: |X-Q|< (1+M)\delta(X)\}.
\end{equation}
The following lemma will be used in Section \ref{mainresult}.
\begin{lemma}[\cite{k1}]\label{1510}
Let $\mu\in A_\infty(d\omega)$, $0\in\Omega$
Then if $Lu=0$ and $0<p<\infty$,
$$\bigg(\int_{\partial\Omega}(S_\alpha(u))^pd\mu\bigg)^{1/p}\le
 C_{\alpha,p}\bigg(\int_{\partial\Omega}(N_\alpha(u))^pd\mu\bigg)^{1/p}.$$
If in addition $u(0)=0$ then
$$\bigg(\int_{\partial\Omega}(N_\alpha(u))^pd\mu\bigg)^{1/p}\le
 C_{\alpha,p}\bigg(\int_{\partial\Omega}(S_\alpha(u))^pd\mu\bigg)^{1/p}.$$
 \end{lemma}

Suppose also that $f$ is a measurable function defined in $\Omega$. For $\alpha>0$ and $Q \in \partial \Omega$, we define 
\begin{equation}\label{tt-square}
 A^{(\alpha)}(f)(Q)= \bigg( \int_{\Gamma_\alpha(Q)} {f(X)}^2 \frac{dX}{{\delta(X)}^{n}}\bigg)^{1/2}.
 \end{equation}
The usual square function of $f$ corresponds to $A(f)=A^{(1)}(f)$. We define the operator $\mathcal{C}(f):\partial \Omega \rightarrow \mathbb{R}$ by
\begin{equation}\label{tt-carleson}
\mathcal{C}(f)(Q):= \sup_{Q \in \Delta}\bigg(\frac{1}{\sigma(\Delta)} \int_{T(\Delta)}f(X)^2\frac{dX}{\delta(X)}\bigg)^{1/2}
 \end{equation}
where $\Delta$ is a surface ball and $T(\Delta)$ is the tent over it.

In the present paper we use the same family of dyadic cubes in $\po$ as the one used in \cite{MPT}. The \emph{shadows} of the dyadic cubes in $\O$ provide a good covering 
of $\O\cap (\po, 4R_0):=\O\cap \{Y\in \R^n :\exists Q_Y\in\partial\Omega\ {\rm{with}}\ |Q_Y-Y|=\delta(Y)\le
 4R_0\}$. To ease the readers task we recall some of their main properties.
 Since $\Omega$ is a CAD in $\R^n$, both $\sigma=\mathcal{H}^{n-1}\res \partial\Omega$ and $\omega_0$ are doubling measures and therefore $(\partial\Omega, |\ |, 
  \sigma)$ and $(\partial\Omega, |\ |, \omega_0)$ are spaces of homogeneous type. M. Christ's construction (see \cite{c}) ensures that there exists a family of dyadic  cubes $\{Q_
  \alpha^k\subset\partial\Omega:k\in\mathbb{Z}, \alpha\in I_k\}$, $I_k\subset \mathbb{N}$ such that for every $k\in\mathbb{Z}$
\begin{equation}\label{tt7.19}
    \sigma(\partial\Omega\setminus\bigcup_{\alpha}Q_\alpha^k)=0, \ \ \ \ \omega_0(\partial\Omega\setminus\bigcup_{\alpha}Q_\alpha^k)=0.
\end{equation}
and the following properties are satisfied:
\begin{enumerate}
\item{} If $l\ge
 k$ then either $Q_\beta^l\subset Q_\alpha^k$ or $Q_\beta^l\cap Q_\alpha^k=\emptyset.$

\item{} For each $(k,\alpha)$ and each $l<k$ there is a unique $\beta$ so that $Q_\alpha^k\subset Q_\beta^l$.

\item{} There exists a constant $C_0>0$ such that ${\rm{diam}}\,Q_\alpha^k\le
 C_08^{-k}$.

\item{} Each $Q_\alpha^k$ contains a ball $B(Z_\alpha^k,8^{-k-1})$.
\end{enumerate}

The Ahlfors regularity property of $\sigma$ combined with properties 3 and 4 ensure that there exists $C_1>1$ such that
\begin{equation}\label{tt7.24}
    C_1^{-1}8^{-k(n-1)}\le
 \sigma(Q_\alpha^k)\le
 C_18^{-k(n-1)}.
\end{equation}
In addition the doubling property of $\omega_0$ yields
\begin{equation}\label{tt7.25}
    \omega_0(B(Z_\alpha^k,8^{-k-1}))\sim\omega_0(Q_\alpha^k).
\end{equation}
For $k\in\mathbb{Z}$ and $\alpha\in I_k$ we define
\begin{equation}\label{tt7.26}
    I_\alpha^k=\{Y\in\Omega: \lambda8^{-k-1}<\delta(Y)<\lambda8^{-k+1}, \ \exists P\in Q_\alpha^k\ \ {\rm{so \ that}}\ \ \lambda8^{-k-1}<|P-Y|<\lambda8^{-k+1}\},
\end{equation}
where $\lambda>0$ is chosen so that for each $k$, the $\{I_\alpha^k\}_{\alpha\in I_k}$'s have finite overlaps and
\begin{equation}\label{tt7.26A}
    \Omega\cap(\partial\Omega,4R_0)\subset \bigcup_{k\le
 k_0,\alpha}I_\alpha^k.
\end{equation}
Here $k_0$ can be selected so that $4R_0<\lambda 8^{-k_0-1}$. We refer the reader to \cite{MPT} for the proof of (\ref{tt7.26A}) and the details on the construction of $\{Q_\alpha^k\}$ and $\{I_\alpha^k\}$.

The various constants that will appear in the sequel may vary from formula to formula, although for simplicity we use the same letter. If we do not give
any explicit dependence for a constant, we mean that it depends
only on the ellipticity constants, CAD constants and the dimension.

\section{Main Result}\label{mainresult}

In this section we state and prove the main result of the present work. Assume that $L_0={\rm{div}}(A_0\nabla\ )$ and $L_1={\rm{div}}(A_1\nabla\ )$ are two symmetric divergence form operators 
operators satisfying (\ref{ellipticity}) defined in a CAD $\Omega$ containing 0. We denote the deviation function of $L_1$ from $L_0$ by 
$$a(X)=\sup\{|A_1(Y)-A_0(Y)|: Y\in B(X,\delta(X)/2)\}$$
and we assume that $L_1$ is a perturbation of $L_0$.
For $t\in [0,1]$ we consider the operators defined by 
\begin{eqnarray}\label{L_t}
L_tu &=& {\rm{div}}(A_t\nabla u ) \\
A_t(X) &=& (1-t)A_0(X)+tA_1(X).
\end{eqnarray}
Note that for each $t$, $L_t$ satisfies (\ref{ellipticity}).
Let $\omega_t$ be the corresponding $L_t-$elliptic measure with pole 0 and let $G_t(0,Y)$ be the Green's function for $L_t$. 
\begin{remark}\label{esch}
Note that since 
$$a_t(X)=\sup\{|A_t(Y)-A_0(Y)|: Y\in B(X,\delta(X)/2)\}= ta(X)$$
then $L_t$ is also a perturbation of $L_0$. Moreover under the assumptions of Corollary \ref{surface}, we have that for every $t\in[0,1]$ $\omega_t$ is a $B_{2(1+\delta_0)}(\sigma)$-weight with a 
uniform $B_{2(1+\delta_0)}$-constant. Thus in particular for $t\in[0,1]$, $\omega_t\in B_{2}(\sigma)$. From now on we assume that $\mathcal C(a)$ is small enough so that the hypothesis of 
Theorem \ref{mainthm1MPT} and those of Corollary \ref{surface} are satisfied (see Proposition \ref{p-fkp}).
\end{remark}

We consider the Dirichlet problems
\begin{equation}\label{mmDir1}
\left\{
  \begin{array}{ll}
    L_tu_t=0 & {\rm{in}} \ \Omega \\
    u_t =f & {\rm{on}}\  \partial\Omega
  \end{array}
\right. \hspace{3em}
\left\{
  \begin{array}{ll}
    L_su_s=0 & {\rm{in}} \ \Omega \\
    u_s =f & {\rm{on}}\  \partial\Omega
  \end{array}
\right.
\end{equation}
for $s,t\in [0,1]$, where $f\in L^2(\sigma)$.
\begin{lemma}\label{lemma3}
Let $\Omega$ be a CAD, $0\in\Omega$. Under the assumptions in Remark \ref{esch}, if $u_t$, $u_s$ are solutions to the Dirichlet problems (\ref{mmDir1}) then
\begin{equation}\label{lemma3-1}
u_s(0)-u_t(0)=(s-t)\int_\Omega \varepsilon(Y)\nabla G_t(0,Y)\nabla u_s(Y)dY
\end{equation}
and
\begin{equation}\label{lemma3-2}
\int_\Omega |\varepsilon(Y)||\nabla G_t(0,Y)||\nabla u_s(Y)|dY\lesssim ||f||_{L^2(\sigma)}.
\end{equation}
In particular
\begin{equation}\label{lemma3-3}
|u_s(0)-u_t(0)|\lesssim ||f||_{L^2(\sigma)}|s-t|.
\end{equation}
\end{lemma}
\begin{proof} Assume that $\delta(0)=4R_0$. Without loss of generality we assume that $A_0=A_1$ on $B(0,R_0)$ and $s>t$. Then integration by parts shows that 
\begin{equation}\label{mm3.7}
u_s(0)-u_t(0)=\int_\Omega G_t(0,Y)L_tu_s(Y)dY=(s-t)\int_\Omega \varepsilon(Y)\nabla G_t(0,Y)\nabla u_s(Y)dY
\end{equation}
which proves (\ref{lemma3-1}). To prove (\ref{lemma3-2}) we proceed as in the proof of Lemma 7.7 in \cite{MPT} using the dyadic surface cubes and their interior \emph{shadows} described in 
Section \ref{prelims}. Assume that $\Omega\setminus (\Omega,R_0)\subset \bigcup_i B(Q_i,2R_0)\cap\Omega$ where $|Q_i-Q_j|\geq R_0$, $Q_i\in\partial \Omega$. Note that the family of balls has
 finite overlap. First we estimate the integral in the tent over $\Delta_0=B(Q_i,2R_0)\cap \partial\Omega$. 
$$\int_{B(Q_i,2R_0)\cap\Omega}|\varepsilon(Y)\nabla G_t(0,Y)\nabla u_s(Y)|dY=\lim_{\delta\rightarrow 0}\int_{T(\Delta_0)\setminus (\partial\Omega,\delta)}|\varepsilon(Y)\nabla G_t(0,Y)\nabla u_s(Y)|dY$$ 
where $T(\Delta_0)=B(Q_i,2R_0)\cap\Omega$. For $\delta>0$ small we compute
\begin{equation}\label{mmI_1}
I_1=\int_{T(\Delta_0)\setminus (\partial\Omega,\delta)}|\varepsilon(Y)\nabla G_t(0,Y)\nabla u_s(Y)|dY \leq \sum_{\substack{Q^k_\alpha\subset 3\Delta_0\\ \delta<\lambda 8^{-k-1}}}\sup_{I^k_\alpha}|\varepsilon(Y)|\int_{I^k_\alpha}|\nabla G_t(0.Y)||\nabla u_s(Y)|dY
\end{equation}
and for $Y\in I_\alpha^k$
$$|\nabla G_t(0.Y)|\lesssim \frac{G_t(0,Y)}{\delta(Y)}\sim\frac{\omega_t(Q_\alpha^k)}{({\rm{diam}}Q_\alpha^k)^{n-1}}$$
thus
\begin{eqnarray}\label{3.8A}
I_1 &\lesssim&\sum_{\substack{Q^k_\alpha\subset 3\Delta_0\\ \delta<\lambda 8^{-k-1}}}\sup_{I^k_\alpha}|\varepsilon(Y)|\frac{\omega_t(Q_\alpha^k)}{({\rm{diam}}Q_\alpha^k)^{n-1}}\int_{I_\alpha^k}|\nabla u_s(Y)|dY\\
&\lesssim & \sum_{\substack{Q^k_\alpha\subset 3\Delta_0\\ \delta<\lambda 8^{-k-1}}}\sup_{I^k_\alpha}|\varepsilon(Y)|\frac{\omega_t(Q_\alpha^k)}{({\rm{diam}}Q_\alpha^k)^{n-1}}\bigg(\int_{I_\alpha^k}|\nabla u_s(Y)|^2\delta(Y)^{2-n}dY\bigg)^{1/2}({\rm{diam}}Q_\alpha^k)^{n/2}({\rm{diam}}Q_\alpha^k)^{n/2-1}  \nonumber\\
&\lesssim & \sum_{\substack{Q^k_\alpha\subset 3\Delta_0\\ \delta<\lambda 8^{-k-1}}}\sup_{I^k_\alpha}|\varepsilon(Y)|\frac{\omega_t(Q_\alpha^k)}{({\rm{diam}}Q_\alpha^k)^{n-1}}\bigg(\int_{I_\alpha^k}|\nabla u_s(Y)|^2\delta(Y)^{2-n}dY\bigg)^{1/2}({\rm{diam}}Q_\alpha^k)^{n-1} \nonumber\\
&\lesssim & \sum_{\substack{Q^k_\alpha\subset 3\Delta_0\\ \delta<\lambda 8^{-k-1}}}\bigg(\int_{I_\alpha^k}\frac{a^2(Y)}{\delta(Y)^n}\cdot\frac{\omega_t(Q_\alpha^k)^2}{({\rm{diam}}Q_\alpha^k)^{2n-2}}dY\bigg)^{1/2}\bigg(\int_{I_\alpha^k}|\nabla u_s(Y)|^2\delta(Y)^{2-n}dY\bigg)^{1/2}({\rm{diam}}Q_\alpha^k)^{n-1} \nonumber\\
&\lesssim & \int_{3\Delta_0}\bigg(\sum\int_{I_\alpha^k}\frac{a^2(Y)}{\delta(Y)^n}\cdot\frac{\omega_t(B(Q_Y,\delta(Y)))^2}{\delta(Y)^{2n-2}}dY\chi_{Q_\alpha^k(Q)}\bigg)^{1/2}
  \bigg(\sum\int_{I_\alpha^k}|\nabla u_s(Y)|^2\delta(Y)^{2-n}dY\chi_{Q_\alpha^k(Q)}\bigg)^{1/2}d\sigma\nonumber\\
  &\lesssim & \int_{3\Delta_0} \bigg(\int_{\Gamma_M(Q)}\frac{a^2(Y)}{\delta(Y)^n}H_t(Y)^2dY\bigg)^{1/2}S_M(u_s)d\sigma(Q)\nonumber\\
   &\lesssim &\bigg(\int_{3\Delta_0}\int_{\Gamma_M(Q)}\frac{a^2(Y)}{\delta(Y)^n}H_t(Y)^2d\sigma\bigg)^{1/2}\bigg(\int_{3\Delta_0}S_M(u_s)^2d\sigma\bigg)^{1/2}\nonumber
\end{eqnarray}
where $H_t(Y)=\frac{\omega_t(B(Q_Y,\delta(Y)))}{\delta(Y)^{n-1}}$. Since $\omega_s\in B_2(\s)$ by Remark \ref{esch} applying Lemma \ref{1510} for $p=2$ and recalling that the $L^2(\sigma)$ norm of the non-tangential maximal function of $u_s$  is bounded by the $L^2(\sigma)$ norm of $f$ we obtain
\begin{equation}\label{mmI_1a}
I_1\lesssim \bigg(\int_{3\Delta_0}\int_{\Gamma_M(Q)}\frac{a^2(Y)}{\delta(Y)^n}H_t(Y)^2d\sigma\bigg)^{1/2}||f||_{L^2(\sigma)} :=D ||f||_{L^2(\sigma)}.
\end{equation}
We now estimate $D$:
\begin{eqnarray}
D^2 &=& \int_{3\Delta_0}\int_{\Gamma_M(Q)}\frac{a^2(Y)}{\delta(Y)^n}H_t(Y)^2dYd\sigma=\int_{3\Delta_0}\int_{\Omega}\chi_{\Gamma_M(Q)}(Y)\frac{a^2(Y)}{\delta(Y)^n}H_t(Y)^2dYd\sigma\nonumber\\
&\lesssim& \int_{12\Delta_0}\bigg(\int_{\Gamma_M(Q)}\frac{a^2(Y)}{\delta(Y)^n}H_t(Y)dY\bigg)d\omega_t\nonumber\\
&\lesssim& \int_{12\Delta_0}NH_t(Q)A(a)(Q)^2d\omega_t\nonumber\\
&\lesssim& \bigg(\int_{12\Delta_0}NH_t(Q)^pd\omega_t\bigg)^{1/p}\bigg(\int_{12\Delta_0}A(a)(Q)^{2q}d\omega_t\bigg)^{1/q}\nonumber\\
\end{eqnarray}
where $\frac{1}{p}+\frac{1}{q}=1$, $p=1+\delta_0$ and $\delta_0$ is selected as in Corollary \ref{surface}. Note also that
$$\bigg(\int_{12\Delta_0}A(a)(Q)^{2q} k_td\sigma\bigg)^{1/q}\leq \bigg(\int_{12\Delta_0}A(a)(Q)^{4q}d\sigma\bigg)^{1/2q}\bigg(\int_{12\Delta_0}k_t^2d\sigma\bigg)^{1/2q}$$
where
$$\bigg(\int_{12\Delta_0}A(a)(Q)^{4q}d\sigma\bigg)^{1/2q}\lesssim \mathcal{C}(a)^2\sigma(12\Delta_0)^{1/2q}$$
and $\mathcal{C}(a)$ is defined in (\ref{tt-carleson}).
Therefore, if $Mk_t$ denotes the maximal function
\begin{eqnarray}\label{mmD}
D^2 &\lesssim& \bigg(\int_{12\Delta_0}k_t^2d\sigma\bigg)^{1/2p}\bigg(\int_{12\Delta_0}Mk_t^{2p}d\sigma\bigg)^{1/2p} \mathcal{C}(a)\sigma(12\Delta_0)^{1/2q}\bigg(\int_{12\Delta_0}k_t^2d\sigma\bigg)^{1/2q}\nonumber\\
&\lesssim&\mathcal{C}(a)\bigg(\fint_{12\Delta_0}k_t^2d\sigma\bigg)^{1/2} \bigg(\fint_{24\Delta_0}k_t^{2p}d\sigma\bigg)^{1/2p}\sigma(12\Delta_0)^{1/2q}\sigma(24\Delta_0)^{1/2+1/2p}\nonumber\\
&\lesssim& \mathcal{C}(a)\bigg(\fint_{12\Delta_0}k_td\sigma\bigg)\bigg(\fint_{12\Delta_0}k_td\sigma\bigg) \sigma(12\Delta_0)\nonumber\\
&\lesssim& \mathcal{C}(a)^2\frac{\omega_t(12\Delta_0)^2}{\sigma(12\Delta_0)}.
\end{eqnarray}

Note that the estimate for $I_1$ is independent of $\delta$ and therefore we conclude combining (\ref{mmI_1}), (\ref{mmI_1a}), (\ref{mmD}) that
\begin{eqnarray}\label{tent}
\int_{B(Q_i,2R_0)\cap\Omega}|\varepsilon(Y)\nabla G_t(0,Y)\nabla u_s(Y)|dY&\lesssim &\mathcal{C}(a)\frac{\omega_t(12\Delta_0)}{\sqrt{\sigma(12\Delta_0)}}\|f\|_{L^2(\sigma)}\nonumber\\
\lesssim \|f\|_{L^2(\sigma)}.
\end{eqnarray}

We now estimate the integral over the complement of the tent in $\Omega$.

\begin{eqnarray}\label{mmI_2}
I_2&=&\int_{\Omega \setminus T(\Delta_0)}|\varepsilon(Y)\nabla G_t(0,Y)\nabla u_s(Y)|dY
\lesssim \int_{\partial\Omega\setminus\frac{1}{4}\Delta_0}S_M(u_s)d\omega_t\\
&\lesssim& \bigg(\int_{\partial\Omega}S_M(u_s)^2d\sigma\bigg)^{1/2}\bigg(\int_{\partial\Omega}k_t^2d\sigma\bigg)^{1/2}
\lesssim ||f||_{L^2(\sigma)}.\nonumber
\end{eqnarray}
Since $\Omega$ is a bounded domain and $\omega_t\in B_2(\sigma)$ then the $L^2(\sigma)$ norm of $k_t$ is bounded in terms of  the $B_2(\sigma)$ norm of $\omega_t$ and $R_0$.
The proof is completed by combining (\ref{tent}) and (\ref{mmI_2}).   
\end{proof}

Let $Q_0\in\partial \Omega$ be fixed and $r>0$. Let $\Delta_r=\Delta_r(Q_0)=\partial\Omega \cap B(Q_0,r)$ and $T(\Delta_r)=T(\Delta_r(Q_0))=\Omega \cap B(Q_0,r)$.

\begin{lemma}\label{lemma4}
Under the assumptions in Remark \ref{esch},  for $f\in L^2(\sigma)$ and  $t\in[0,1]$ consider
$$\Psi(t)=\frac{1}{\omega_t(\Delta_r)}\int_{\Delta_r}fk_td\sigma.$$
Then $\Psi(t)$ is Lipschitz. Moreover
$$\dot{\Psi}(t)=\frac{1}{\omega_t(\Delta_r)}\int_{\Delta_r}\dot{k}_t\bigg(f-\fint_{\Delta_r}fd\omega_t\bigg)d\sigma$$
where $\dot{k}_t$ exists as the weak $L^2$ limit of $k_{t+h}-k_t/h$ as $h$ tends to zero.
\end{lemma}
\begin{proof} Let $s,t\in[0,1]$ and denote by $u_t$, $u_s$ the solutions of (\ref{mmDir1}) for a given $f\in L^{2}(\sigma)$. Then 
\begin{eqnarray}\label{PsiLip}
\Psi(t)-\Psi(s) &=&\frac{1}{\omega_t(\Delta_r)}\int_{\Delta_r}k_tfd\sigma-\frac{1}{\omega_s(\Delta_r)}\int_{\Delta_r}k_sfd\sigma\nonumber\\ 
&=& \frac{1}{\omega_t(\Delta_r)}\bigg(\int_{\Delta_r}k_tfd\sigma-\int_{\Delta_r}k_sfd\sigma\bigg)+\bigg(\frac{1}{\omega_t(\Delta_r)}-\frac{1}{\omega_s(\Delta_r)}\bigg)\int_{\Delta_r}k_sfd\sigma\nonumber
\end{eqnarray}
which will show that $\Psi$ is Lipschitz due to (\ref{lemma3-3}). In particular we have that
$$\bigg|\frac{u_{t+h}-u_t}{h}\bigg|=\bigg|\int_{\partial \Omega}\frac{k_{t+h}-k_t}{h}fd\sigma\bigg|\lesssim ||f||_{L^2(\sigma)}$$
which shows that $\dot{k}_t$ exists as a weak $L^2$ limit of $k_{t+h}-k_t/h$ as $h$ tends to zero. In order to compute $\dot{\Psi}(t)$ we write
\begin{eqnarray}\label{Psidot}
\frac{\Psi(t+h)-\Psi(t)}{h}&=&\frac{1}{\omega_{t+h}(\Delta_r)}\int_{\Delta_r}\frac{k_{t+h}-k_t}{h}fd\sigma+\frac{1}{h}\bigg(\frac{1}{\omega_{t+h}(\Delta_r)}-\frac{1}{\omega_{t}(\Delta_r)}\int_{\Delta_r}k_tfd\sigma\bigg)\nonumber\\ 
&=& \frac{1}{\omega_{t+h}(\Delta_r)}\int_{\Delta_r}\frac{k_{t+h}-k_t}{h}fd\sigma+\frac{\omega_{t}(\Delta_r)-\omega_{t+h}(\Delta_r)}{h}\frac{1}{\omega_{t+h}(\Delta_r)\omega_{t}(\Delta_r)}\int_{\Delta_r}k_tfd\sigma\nonumber
\end{eqnarray}
thus
$$\dot{\Psi}(t)=\frac{1}{\omega_t(\Delta_r)}\int_{\Delta_r}\dot{k}_t\bigg(f-\fint_{\Delta_r}fd\omega_t\bigg)d\sigma.$$
\end{proof}
\begin{lemma}\label{lemma5}
Under the assumptions in Remark \ref{esch}, there exist positive constants $\beta$, $\gamma<1$ and $C$ such that if $\Psi(t)$ is the function defined in Lemma \ref{lemma4} where $f$ is a 
non-negative function with $spt(f)\subset \Delta_r(Q_0)$ and $||f||_{L^2(d\sigma /\sigma(\Delta(Q_0,r)))}\leq 1$ then
\begin{equation}\label{estdotpsi}
|\dot{\Psi}(t)|\leq C\bigg[r^\gamma+\sup_{\substack{Q\in\partial \Omega\\ s\leq r^\beta}}\bigg(\frac{1}{\sigma(\Delta_s(Q))}\int_{T(\Delta_s(Q))}\frac{a(Y)^2}{\delta(Y)}dY\bigg)^{1/2}\bigg]
\end{equation} 
for $0\leq t\leq 1$.
\end{lemma}
\begin{proof} Assume that $u_t$ is the solution of the problem 
\begin{equation}
\left\{
  \begin{array}{ll}
    L_tu_t=0 & {\rm{in}} \ \Omega \\
    u_t =h_t & {\rm{on}}\  \partial\Omega
  \end{array}
\right. 
\end{equation}
where $$h_t=\frac{1}{\omega_t(\Delta_r)}(f-\fint_{\Delta_r}fd\omega_t)\chi_{\Delta_r}$$
and $\Delta_r=\Delta_r(Q)$. As in (\ref{lemma3-1}) and Lemma \ref{lemma4} we have that 
$$\dot{\Psi}(t)\leq \int_\Omega |\varepsilon(Y)||\nabla_Y G_t||\nabla u_t|dY.$$
We prove Lemma \ref{lemma5} using the following three claims.

\textit{Claim 1. For fixed $Q_0\in \partial \Omega$, $r>0$ and $\Delta_{Mr}=\Delta_{Mr}(Q_0)$, $M>0$
\begin{equation}\label{claim1}
\int_{T(\Delta_{Mr})} |\varepsilon(Y)||\nabla_Y G_t||\nabla u_t|dY\lesssim \sup_{\substack{Q\in\partial \Omega\\ s\leq r^\beta}}\bigg(\frac{1}{\sigma(\Delta_s(Q))}\int_{T(\Delta_s(Q))}\frac{a(Y)^2}{\delta(Y)}dY\bigg)^{1/2}
\end{equation}
where $\beta$ is a given positive constant.}

\textit{Proof of Claim 1.} To prove (\ref{claim1}) we proceed as in the proof of (\ref{3.8A}) in Lemma \ref{lemma3}. In a similar manner we obtain the analog of (\ref{mmI_1a})-(\ref{tent}) which in this case yield
$$\int_{T(\Delta_{Mr})} |\varepsilon(Y)||\nabla_Y G_t||\nabla u_t|dY \lesssim \mathcal{C}(a)\frac{\omega_t(\Delta_r)}{\sqrt{\sigma(\Delta_r)}}||h_t||_{L^2}\lesssim \sup_{\substack{Q\in\partial \Omega\\ s\leq r^\beta}}\bigg(\frac{1}{\sigma(\Delta_s(Q))}\int_{T(\Delta_s(Q))}\frac{a(Y)^2}{\delta(Y)}dY\bigg)^{1/2}$$ 
due to the selection of the boundary data $h_t$.

\textit{Claim 2. Let $0<r<8R_0$ and $R_0$ is selected as in Lemma \ref{lemma3}. For fixed $Q_0\in\partial\Omega$,  there exists a constant $\eta>0$ such that
\begin{equation}\label{claim2}
\int_{\Omega\setminus T(\Delta_{8R_0})\cap (\partial\Omega,4R_0)} |\varepsilon(Y)||\nabla_Y G_t||\nabla u_t|dY\lesssim r^\eta
\end{equation}
where $\Delta_{8R_0}=\Delta_{8R_0}(Q_0)$.}

\textit{Proof of Claim 2.} Note that $\varepsilon(Y)\equiv 0$ in $B(0,\frac{\delta(0)}{4})$ where $G_t(0,-)$ denotes the Green's function of $L_t$ with pole at $0$.  We denote by $\Gamma_{R_0}=\Omega\setminus T(\Delta_{8R_0})\cap (\partial\Omega,4R_0)$ and apply Schwartz's inequality to obtain  
\begin{equation}\label{Claim2-a}
\int_{\Gamma_{R_0}} |\varepsilon(Y)||\nabla_Y G_t||\nabla u_t|dY \lesssim \sup_{(\partial\Omega,4R_0)}|\varepsilon(Y)|\bigg(\int_{\Gamma_{R_0}}|\nabla u_t|^2 dY\bigg)^{1/2}\bigg(\int_{(\partial\Omega,4R_0)}|\nabla_Y G_t|^2dY\bigg)^{1/2}.
\end{equation}
In addition,
\begin{equation}\label{mmG_t}
\bigg(\int_{(\partial\Omega,4R_0)}|\nabla_Y G_t|^2dY\bigg)^{1/2}\lesssim R_0^{n-2/2}G_t(0,A_{R_0}(Q_0))\lesssim  R_0^{\frac{n-2}{2}}\frac{\omega^{}(\Delta_{R_0}(Q_0))}{R_0^{n-2}}\lesssim R_0^{-\frac{n-2}{2}}
\end{equation}
where $A_{R_0}(Q_0)=(1-R_0)Q_0$. We will now estimate $\sup|u_t|$. Note that 
$$u_t=\int_{\partial\Omega}h_td\omega_t^X=\int_{\partial\Omega}h_tK_t(X,Q)d\omega_t$$
where $K_t(X,Q)=\frac{d\omega_t^X}{d\omega_t}(Q)$ and $d\omega_t=k_td\sigma$. For $\frac{1}{4}|X-Q_0|>|Q-Q_0|$ we have
\begin{eqnarray}\label{claim3-1-null}
|K_t(X,Q)-K_t(X,Q_0)| &\lesssim&  \bigg(\frac{|Q-Q_0|}{|X-Q_0|}\bigg)^\eta\frac{G_t(X,A_{|X-Q_0|}(Q_0))}{G_t(0,A_{|X-Q_0|}(Q_0))}\nonumber\\ 
&\lesssim&  \bigg(\frac{|Q-Q_0|}{|X-Q_0|}\bigg)^\eta\frac{\omega_t^X(\Delta_{|X-Q_0|}(Q_0))}{\omega_t(\Delta_{|X-Q_0|}(Q_0))}\nonumber\\ 
&\lesssim&  \bigg(\frac{|Q-Q_0|}{|X-Q_0|}\bigg)^\eta\frac{1}{\omega_t(\Delta_{|X-Q_0|}(Q_0))}\nonumber 
\end{eqnarray}
By its definition $h_t=0$ outside $\Delta_r$ therefore, 
$$u_t(X)=\int_{\partial\Omega}(K_t(X,Q)-K_t(X,Q_0))h_tk_td\sigma$$
and
\begin{equation}\label{star}
|u_t(X)|\lesssim \bigg(\frac{r}{|X-Q_0|}\bigg)^\eta\frac{1}{\omega_t(\Delta_{|X-Q_0|}(Q_0))}
\end{equation}
for some $\eta>0$.
We cover $\Gamma_{R_0}$ by balls of radius $8R_0$ such that the balls of radius $2R_0$ are disjoint and do not intersect with $T(\Delta_{2R_0})$. Using Cacciopoli's inequality, the maximum principle and (\ref{star}), we have
\begin{equation}\label{Claim2-bA}
\bigg(\int_{B(Q_l,2R_0)}|\nabla u_t|^2 dY\bigg)^{1/2} \lesssim R_0^{-1+n/2}\sup_{B(Q_l,2R_0)} \lesssim R_0^{-1+n/2}\bigg(\frac{r}{R_0}\bigg)^\eta \frac{1}{\omega_t(\Delta_{R_0}(Q_0))},
\end{equation} 
which shows that
\begin{equation}\label{Claim2-b}
\bigg(\int_{\Gamma_{R_0}}|\nabla u_t|^2 dY\bigg)^{1/2} \lesssim r^\eta
\end{equation}
with constant depending also on $\text{diam}\Omega$ and $R_0$. The claim follows by combining (\ref{Claim2-a}), (\ref{mmG_t}) and (\ref{Claim2-b}).


\textit{Claim 3. For fixed $Q_0\in \partial \Omega$, let $0<r<R<8R_0<\frac{\delta(0)}{4}$. Then 
\begin{equation}\label{claim3}
\int_{T(\Delta_R)\setminus T(\Delta_r)} |\varepsilon(Y)||\nabla_Y G_t||\nabla u_t|dY\leq \sum_{j=1}^L\bigg(\int_{T(\Delta_{8^{j}r})\setminus T(\Delta_{8^{j-1}r})} |\varepsilon(Y)||\nabla_Y G_t||\nabla u_t|dY\bigg)\lesssim \mathcal{C}(a)
\end{equation}
where $\Delta_r=\Delta_r(Q_0)$, $\Delta_R=\Delta_R(Q_0)$ and $L$ is chosen such that $8^Lr\leq R<8^{L+1}r$.}

\textit{Proof of Claim 3.} We will estimate
$$I_j=\int_{T(\Delta_{8^{j}r})\setminus T(\Delta_{8^{j-1}r})} |\varepsilon(Y)||\nabla_Y G_t||\nabla u_t|dY$$
for fixed $j$. We start by defining a dyadic decomposition on $A_j=T(\Delta_{8^{j}r})\setminus T(\Delta_{8^{j-1}r})$. Cover $\Delta_j'=\Delta_{8^{j}r}\setminus \Delta_{8^{j-1}r}$ by balls $B_i(Q_i,\rho)$ with center $Q_i\in\Delta'_j$ and radius $\rho=8^{j-6}r$. The numbers of the balls needed is roughly $c_n=8^{7n-7}-8^{4n-3}$. In that case we are able to cover small strips close to the boundary by balls. Then we split 
$$A_j=[(\bigcup^{c_n}_{i=1} B_i)\cap A_j]\cup[A_j\setminus\bigcup^{c_n}_{i=1} B_i]=V_j\cup W_j.$$
Following the pattern in the proof of Lemma 7.7 in \cite{MPT} we will estimate first the term close to the boundary
$$\int_{V_j} |\varepsilon(Y)||\nabla_Y G_t||\nabla u_t|dY\leq \lim_{\epsilon\rightarrow 0}\int_{V_j\setminus (\partial\Omega,\epsilon)} |\varepsilon(Y)||\nabla_Y G_t||\nabla u_t|dY=\lim_{\epsilon\rightarrow 0} I^\epsilon_j$$
and
\begin{eqnarray}\label{claim3-1}
I^\epsilon_j &\lesssim&  \sum_{\substack{Q_\alpha^k\in 3\Delta'_j\\ \text{diam}Q_\alpha^k\leq 8^jr}} \sup_{I^k_\alpha}|\varepsilon(Y)|\int_{I^k_\alpha}|\nabla G_t(0,Y)||\nabla u_t(Y)|dY\\ 
&\lesssim& \sum_{\substack{Q_\alpha^k\in 3\Delta'_j\\ \text{diam}Q_\alpha^k\leq 8^jr}}\bigg(\int_{I^k_\alpha}\frac{a^2(Y)G_t(0,Y)^2}{\delta(Y)^2}dY\bigg)^{1/2}\bigg(\int_{I^k_\alpha}|\nabla u_t|^2dY\bigg)^{1/2}.\nonumber
\end{eqnarray} 
Now for $Y\in I_\alpha^k$
\begin{equation}\label{rev1-1}
G_t(0,Y)\sim\frac{\omega_t(Q_\alpha^k)}{(\text{diam}Q_\alpha^k)^{n-2}}
\end{equation}
and 
\begin{equation}\label{rev1-2}
\int_{I_\alpha^k}|\nabla u|^2dY\lesssim (\text{diam}Q_\alpha^k)^{-2}\int_{2I_\alpha^k}u_t^2dY.
\end{equation}
On the other hand, for $Y\in 2I_\alpha^k$
\begin{equation}\label{rev1-3}
|u_t(Y)|\lesssim\bigg(\frac{\text{diam}Q_\alpha^k}{8^jr}\bigg)^{\eta}\sup_{A_j}|u_t|
\end{equation}
for some $\eta>0$. We will now estimate $\sup_{A_j}|u_t|$. In particular, for $Z\in A_j$ we have
\begin{eqnarray}\label{rev1-4}
|u_t(Z)| &\leq&  \frac{1}{\omega_t(\Delta_r)}\int_{\Delta_r}\bigg|f-\fint_{\Delta_r} fd\omega_t\bigg|d\omega_t^Z\\
&\leq& \frac{1}{\omega_t(\Delta_r)}\int_{\Delta_r}|f|d\omega_t^Z+\frac{\omega_t^Z(\Delta_r)}{\omega_t(\Delta_r)^2}\int_{\Delta_r}|f|d\omega_t\nonumber
\end{eqnarray}
and 
$$\int_{\Delta_r}|f|d\omega_t^Z=\int_{\Delta_r}K_t(Z,Q)|f|d\omega_t$$
where
$$K_t(Z,Q)\sim\frac{G_t(Z,A_r(Q_0))}{G_t(0,A_r(Q_0))}\sim\frac{\omega^Z_t(\Delta_r)}{\omega_t(\Delta_r)}.$$
Therefore (\ref{rev1-4}) becomes
\begin{eqnarray}\label{rev1-5}
|u_t(Z)| &\lesssim&  \frac{\omega_t^Z(\Delta_r)}{\omega_t(\Delta_r)^2}\int_{\Delta_r}|f|d\omega_t\\
&\lesssim& \sigma(\Delta_r)\frac{\omega_t^Z(\Delta_r)}{\omega_t(\Delta_r)^2}\bigg(\fint_{\Delta_r}|f|d\omega_t\bigg)^{1/2}\bigg(\fint_{\Delta_r}k_t^2d\sigma\bigg)^{1/2}\nonumber\\
&\lesssim& \frac{\omega_t^Z(\Delta_r)}{\omega_t(\Delta_r)}\lesssim\bigg(\frac{\delta(Z)}{8^jr}\bigg)^\eta\frac{\omega_t^{P_j}(\Delta_r)}{\omega_t(\Delta_r)}\nonumber
\end{eqnarray}
for some $\eta>0$ and $P_j\in W_j$. Now
$$\frac{\omega_t(\Delta_r)}{\omega_t(\Delta_{8^jr})}\sim \frac{\omega_t^{P_j}(\Delta_r)}{\omega_t^{P_j}(\Delta_{8^jr})}$$
thus from (\ref{rev1-5}), we obtain
$$|u_t(Z)| \lesssim \bigg(\frac{\delta(Z)}{8^jr}\bigg)^\eta\frac{1}{\omega_t(\Delta_{8^jr})}\lesssim \frac{1}{\omega_t(\Delta_{8^jr})}$$
and (\ref{rev1-3}) becomes
\begin{equation}\label{rev1-6}
|u_t(Y)|\lesssim \bigg(\frac{\text{diam}Q_\alpha^k}{8^jr}\bigg)^\eta\frac{1}{\omega_t(\Delta_{8^jr})}
\end{equation}
for $Y\in 2I_\alpha^k$, where for simplicity we used the same notation for the exponent $\eta>0$. 

We now return to the estimate of $I_j^\varepsilon$ in (\ref{claim3-1}) to obtain
\begin{eqnarray}\label{claim3-1A}
I_j^\epsilon &\lesssim& \sum_{\substack{Q_\alpha^k\in 3\Delta'_j\\ \text{diam}Q_\alpha^k\leq 8^jr}} \bigg(\int_{I^k_\alpha}\frac{a^2(Y)}{\delta(Y)}dY\bigg)^{1/2}\bigg(\frac{1}{\text{diam}Q_\alpha^k}\bigg)^{1/2+1}\frac{\omega_t(Q_\alpha^k)}{(\text{diam}Q_\alpha^k)^{n-2}}\nonumber\\
&\cdot&\bigg(\frac{\text{diam}Q_\alpha^k}{8^jr}\bigg)^\eta\frac{1}{\omega_t(\Delta_{8^jr})}(\text{diam}Q_\alpha^k)^{n/2}\nonumber\\
&\lesssim& \bigg(\frac{1}{8^jr}\bigg)^\eta\frac{1}{\omega_t(\Delta_{8^jr})}\sum_{\substack{Q_\alpha^k\in 3\Delta'_j\\ \text{diam}Q_\alpha^k\leq 8^jr}} \bigg(\frac{1}{\sigma(Q_\alpha^k)}\int_{I^k_\alpha}\frac{a^2(Y)}{\delta(Y)}dY\bigg)^{1/2}\omega_t(Q_\alpha^k)(\text{diam}Q_\alpha^k)^\eta\nonumber\\
&\lesssim& \bigg(\frac{1}{8^jr}\bigg)^\eta\frac{1}{\omega_t(\Delta_{8^jr})}\mathcal{C}(a )\sum_{\substack{Q_\alpha^k\in 3\Delta'_j\\ \text{diam}Q_\alpha^k\leq 8^jr}}\omega_t(Q_\alpha^k)(\text{diam}(Q_\alpha^k)^\eta\nonumber\\
&\lesssim& 8^{-2j\eta}\mathcal{C}(a).
\end{eqnarray}

Finally to estimate the integral over $W_j$, we cover $W_j$ with balls $B_{jl}$ with centers $Q_{jl}\in W_j$ and radius $\rho_{j}=8^{j-11}r$. Following the pattern in the proof above we have that for $Y\in B_{jl}$
$$G_t(0,Y)\sim\frac{\omega_t(\Delta_{jl})}{(8^jr)^{n-2}}$$
and
$$\bigg(\int_{B_{jl}}|\nabla u_t|^2dY\bigg)^{1/2}\leq (8^jr)^{\frac{n-2}{2}}\bigg(\frac{r}{8^jr}\bigg)^\eta\frac{1}{\omega_t(\Delta_{jl})}.$$
Therefore,
\begin{eqnarray}\label{claim3-2}
\int_{W_j} |\varepsilon(Y)||\nabla_Y G_t||\nabla u_t|dY &\lesssim&  \sum_l\sup_{B_{jl}}|\varepsilon(Y)|\bigg(\int_{B_{jl}}\frac{G_t(0,Y)^2}{\delta(Y)^2}dY\bigg)^{1/2}\bigg(\int_{B_{jl}}|\nabla u_t|^2dY\bigg)^{1/2}\nonumber\\
&\lesssim& \sum_l\bigg(\int_{B_{jl}}\frac{a^2(Y)G_t(0,Y)^2}{\delta(Y)^2}dY\bigg)^{1/2}\bigg(\int_{B_{jl}}|\nabla u_t|^2dY\bigg)^{1/2}\nonumber\\
&\lesssim& \mathcal{C}(a)\sum_l\frac{1}{(8^jr)^{1/2}}\frac{\omega_t(\Delta_{jl})}{(8^jr)^{n-2}}(8^jr)^{\frac{n}{2}-1}8^{-j\eta}\frac{1}{\omega_t(\Delta_{jl})}(8^jr)^{\frac{n-1}{2}}\nonumber\\
&\lesssim& 8^{-\eta j}\mathcal{C}(a)
\end{eqnarray}
 for some $\eta>0$,  which concludes the proof of Claim 3.

To finish the proof of Lemma \ref{lemma5}, we write 
$$\Omega=T(\Delta_{Mr})\cup\bigg(\Omega\setminus T(\Delta_{r^\beta})\bigg)\cup \bigg(T(\Delta_{r^\beta})\setminus T(\Delta_{Mr})\bigg)$$
and combine Claims 1, 2 and 3.
\end{proof}


\begin{thm}\label{main}
Let $\Omega$ be a CAD, there exist $\eta_0>0$ and  $q_0>2$ such that if $\mathcal C(a)<\eta_0$ and $\omega_0\in B_{q_0}(\sigma)$ then  
\begin{equation}\label{mainres}
\bigg(\fint_{\Delta_r}k_1^2d\sigma\bigg)^{1/2}\leq \bigg[\frac{(\fint_{\Delta_r}k_0^2d\sigma)^{1/2}}{\fint_{\Delta_r}k_0d\sigma}+Cr^\gamma+C\sup_{\substack{Q\in\partial \Omega\\ s\leq r^\beta}}\bigg(\frac{1}{\sigma(\Delta_s(Q))}\int_{T(\Delta_s(Q))}\frac{\a(Y)^2}{\delta(Y)}dY\bigg)^{1/2}\bigg]\fint_{\Delta_r}k_1d\sigma
\end{equation}   
where the positive constants $\beta$, $\gamma$ and $C$ only depend on the CAD constants, the ellipticity constants and the dimension.

\end{thm}
\begin{proof} By (\ref{estdotpsi}) and the fundamental theorem of calculus we have
$$\Psi(1)\leq \Psi(0)+e(r)$$
where
$$e(r)=C\bigg(r^\gamma+\sup_{\substack{Q\in\partial \Omega\\ s\leq r^\beta}}\bigg(\frac{1}{\sigma(\Delta_s(Q))}\int_{T(\Delta_s(Q))}\frac{\a(Y)^2}{\delta(Y)}dY\bigg)^{1/2}\bigg)$$
and $$\Psi(s)=\frac{1}{\omega_s(\Delta_r)}\int_{\Delta_r}fk_sd\sigma$$ for $f\geq 0$, $\text{spt}f\subset \Delta_r$ and $||f||_{L^2(d\sigma/\Delta_r)}\leq 1$. By the Cauchy-Schwartz inequality we obtain
\begin{eqnarray}\label{fund}
\fint_{\Delta_r}fk_1d\sigma&\leq& \bigg[\frac{1}{\omega_0(\Delta_r)}\int_{\Delta_r}fk_0d\sigma+e(r)\bigg]\fint_{\Delta_r}k_1d\sigma\nonumber\\
&\leq& \bigg[\frac{1}{\omega_0(\Delta_r)}\bigg(\fint_{\Delta_r}f^2d\sigma\bigg)^{1/2}\bigg(\int_{\Delta_r}k_0^2d\sigma
\bigg)^{1/2}+e(r)\bigg]\fint_{\Delta_r}k_1d\sigma\nonumber\\
&\leq& \bigg[\frac{\bigg(\fint_{\Delta_r}k_0^2d\sigma\bigg)^{1/2}}{\fint_{\Delta_r}k_0d\sigma}+e(r)\bigg]\fint_{\Delta_r}k_1d\sigma.\nonumber
\end{eqnarray}
By duality we have
$$\bigg(\fint_{\Delta_r}k_1^2d\sigma\bigg)^{1/2} \leq\bigg[\frac{\bigg(\fint_{\Delta_r}k_0^2d\sigma\bigg)^{1/2}}{\fint_{\Delta_r}k_0d\sigma}+e(r)\bigg]\fint_{\Delta_r}k_1d\sigma$$
and the proof is complete.
\end{proof}

The following result is an easy corollary of Theorem \ref{main}
\begin{cor}\label{cor1}
Let $\Omega$ be a CAD.
Assume that $\log k_0\in VMO(\sigma)$ and that $L_1$ is a perturbation of $L_0$ whose deviation from $L_0$ has vanishing Carleson constant, then given $\varepsilon>0$ there exists $r_0>0$ such that for every $r\leq r_0$
$$\bigg(\fint_{\Delta_r}k_1^2d\sigma\bigg)^{1/2}\leq (1+\varepsilon)\fint_{\Delta_r}k_1d\sigma.$$
\end{cor}

In \cite{kt1} the authors proved that the logarithm of the Poisson kernel on a chord arc domain with vanishing constant belongs to $\rm{VMO}(\sigma)$. Thus Corollary \ref{cor1} yields:
\begin{cor}\label{cor2}
If $\Omega$ is a chord arc domain with vanishing constant and $L_1$ is a perturbation of the Laplacian whose deviation from the Laplacian has vanishing Carleson constant then $\log k_1\in VMO(\partial \Omega)$.
\end{cor}

\bigskip
\noindent {\bf {Acknowledgments:}} J. Pipher was partially supported by NSF DMS grant 0901139. T. Toro was partially supported by NSF DMS grant 0856687. E. Milakis was partially supported 
by Marie Curie International Reintegration Grant No 256481 within the 7th European Community Framework Programme. 
Part of this work was carried out while the first author was visiting the University of Washington. He wishes to thank the Department of Mathematics for the warm hospitality and support.

\begin{tabular}{l}
Emmanouil Milakis\\ University of Cyprus \\ Department of Mathematics \& Statistics \\ P.O. Box 20537\\
Nicosia, CY- 1678 CYPRUS
\\ {\small \tt emilakis@ucy.ac.cy}
\end{tabular}
\begin{tabular}{lr}
Jill Pipher\\ Brown University \\ Mathematics Department\\ Box 1917 \\
Providence, RI 02912  USA
\\ {\small \tt jpipher@math.brown.edu}
\hfill
\end{tabular}
\begin{tabular}{lr}
Tatiana Toro \\ University of Washington \\ Department of Mathematics \\ Box 354350 \\
Seattle, WA 98195-4350 USA
\\ {\small \tt toro@math.washington.edu}
\end{tabular}

\end{document}